\documentclass{article}
\usepackage[utf8]{inputenc} 
\usepackage[T1]{fontenc}    
\usepackage{url}            
\usepackage{booktabs}       
\usepackage{amsfonts}       
\usepackage{nicefrac}       
\usepackage{microtype}      

\usepackage{enumitem}

\usepackage{algorithm}
\usepackage{verbatim}

\usepackage[letterpaper, left=1in, right=1in, top=1in, bottom=1in]{geometry}

\PassOptionsToPackage{hypertexnames=false}{hyperref}  
\usepackage{parskip}

\usepackage[dvipsnames]{xcolor}
\usepackage[colorlinks=true, linkcolor=blue!100!black, citecolor=blue!100!black,urlcolor=black]{hyperref}
\usepackage{microtype}
\usepackage{hhline}

\usepackage{algorithm}

\usepackage{natbib}
\bibpunct{(}{)}{;}{a}{,}{,}

\usepackage{amsthm}
\usepackage{mathtools}
\usepackage{amsmath}
\usepackage{bbm}
\usepackage{amsfonts}
\usepackage{amssymb}
\let\vec\undefined

\usepackage{xpatch}

\theoremstyle{definition}  

\newtheorem{lemma}{Lemma}[section]

\newtheorem{corollary}{Corollary}[section]

\theoremstyle{plain}
\newtheorem{remark}{Remark}

\newtheorem{theorem}{Theorem}[section]

\xpatchcmd{\proof}{\itshape}{\normalfont\proofnameformat}{}{}
\newcommand{\proofnameformat}{\bfseries}

\newcommand{\E}{{\mathbb E}}

\newcommand{\R}{{\mathbb R}}

\newcommand{\PP}{\mathbb{P}}

\renewcommand{\b}[1]{\boldsymbol{\mathbf{#1}}}
\newcommand{\vec}[1]{\b{#1}}

\newcommand{\eps}{\epsilon}

\newcounter{rcnt}[section]

\def\argmin{\mathop{\rm argmin}}
\def\argmax{\mathop{\rm argmax}}
\newcommand{\QQ}{\mathbb Q}

\newcommand{\cA}{{\mathcal A}}

\newcommand{\cD}{{\mathcal D}}
\newcommand{\cE}{{\mathcal E}}
\newcommand{\cF}{{\mathcal F}}
\newcommand{\cG}{{\mathcal G}}
\newcommand{\cH}{{\mathcal H}}

\newcommand{\cM}{{\mathcal M}}
\newcommand{\cN}{{\mathcal N}}

\newcommand{\cR}{{\mathcal R}}

\newcommand{\cW}{{\mathcal W}}
\newcommand{\cX}{{\mathcal X}}

\newcommand{\Xn}{\mathbf{x}_n}
\newcommand{\GaussAvgsF}{\cW(\cF)}
\newcommand{\GaussAvgs}{\cW}
\newcommand{\EmpGaussAvgs}{\widehat{\cW}}

\newcommand{\inner}[1]{\langle #1 \rangle}
\newcommand{\norm}[1]{\|#1\|}
\newcommand{\bxi}{\boldsymbol{\xi}}
\newcommand{\by}{\boldsymbol{y}}
\newcommand{\reals}{\mathbb{R}}

\title{On the Minimal Error of Empirical Risk Minimization}

\author{Gil Kur\\ MIT \and Alexander Rakhlin \\MIT}
\date{}

\begin{document}

\maketitle

\begin{abstract}
     We study the minimal  error of the Empirical Risk Minimization (ERM) procedure in the task of regression, both in the random and the fixed design settings. Our sharp lower bounds shed light on the possibility (or impossibility) of adapting to simplicity of the model generating the data. In the fixed design setting, we show that the error is governed by the global complexity of the entire class. In contrast, in random design, ERM may only adapt to simpler models if the local neighborhoods around the regression function are nearly as complex as the class itself, a somewhat counter-intuitive conclusion. We provide sharp lower bounds for performance of ERM for both Donsker and non-Donsker classes. We also discuss our results through the lens of recent studies on interpolation in overparameterized models.
\end{abstract}

\section{Introduction}

An increasing number of machine learning applications employ flexible overparameterized models to fit the training data. Theoretical analysis of such `overfitted' solutions has been a recent focus of the learning community. It is conjectured that the use of large overparameterized neural networks makes the loss landscape amenable to optimization through local search methods, such as stochastic gradient descent. It is also hypothesized that implicit regularization, arising from the choice of the optimization algorithm and the neural network architecture, mitigates the large complexity and ensures that the `overfitted' solutions generalize. 

Suppose a `simple' class $\cH$ of models captures the relationship between the covariates $X$ and the response variable $Y$. Inspired by the use of overparameterized models, we may take a much larger class $\cF\supset \cH$ for computational or other purposes (such as lack of explicit description of $\cH$) and minimize training loss over this larger class. It is natural to ask whether the learning procedure can \textit{adapt} to the fact that data comes from a simple model $f_0\in\cH$, in the sense that the prediction error depends on the statistical complexity of $\cH$ rather than $\cF$. We do have positive examples of this type: the least squares solution (that is, empirical risk minimization with respect to square loss) over the class of all convex functions $\cF$ on a convex compact subset of $\mathbb{R}^d$ (with $d\leq 4$) automatically enjoys the faster ``parametric'' rate $\Tilde{O}(k/n)$ of convergence to the true regression function $f_0\in\cH$ if $f_0$ is a piece-wise linear convex function with $k$ pieces. This rate should be contrasted with the slow non-parametric rate $\Theta(n^{-4/(d+4)})$ when the true regression function is `complex' and cannot be approximated well by a piece-wise linear convex function. 

How generic is this phenomenon of automatic adaptivity of empirical minimizers to simplicity of the true model? An affirmative answer would lend credibility to the practice of taking large models, whereas a negative answer would necessitate the study of conditions that can make such adaptivity possible.

This paper studies the fundamental limits of adaptivitiy of empirical risk minimization (ERM) in the setting of nonparametric regression (or, prediction with square loss and a well-specified model), in both random and fixed design. In contrast with the standard minimax approach to lower bounds, which may hide the true performance of ERM on simple models, we focus on lower bounds that hold for \textit{any} (rather than the worst-case) regression function in a given class. In the fixed design setting, we show that---informally speaking---for rich classes $\cF$, dependence on the global statistical complexity of the class is unavoidable, as it controls the error of ERM for any true regression function $f_0$, no matter how `simple' it is. In contrast, in the random design case, the situation is more subtle. Somewhat counter-intuitively, we show that for rich classes $\cF$, adaptation to the simplicity of $f_0$ may only be possible if the local neighborhood of $f_0$ in $\cF$ is nearly as rich as the class $\cF$. This finding can be viewed through the lens of recent results on interpolation \citep{belkin2019does,belkin2018overfitting,bartlett2020benign,liang2020just}. In these papers, the solutions can be seen as `simple-plus-spiky' \citep{wyner2017explaining} with spikes responsible for fitting the training data without affecting the error with respect to the population. Since in these models there are enough degrees of freedom to fit any noisy data, the effective function classes have rich local neighborhoods. In such cases, it is still possible that `overfitting' to the training data does not result it large out-of-sample error. Conversely, we show that---again, informally speaking---if $f_0$ is embedded in a local neighborhood in $\cF$ with low complexity, the empirical minimizer will necessarily be attracted to a solution far away from $f_0$ with respect to the out-of-sample loss. This finding initially appeared counter-intuitive to the authors.

\section{Formal Model}
\label{sec:formal}

We now present the formal model. Let $\cF$ be a \textit{convex} class of real-valued functions on some domain $\cX$. We aim to recover $f_0 \in \cF$ based on $n$ samples $Y_i=f_0(X_i)+\xi_i$, $i=1,\ldots,n$, under the assumption $f_0 \in\cF$ and  $\xi_1,\ldots,\xi_n \overset{i.i.d.}{\sim} N(0,1)$. In the random design setting, $X_1,\ldots,X_n\overset{i.i.d.}{\sim} \PP$, where $\PP$ is some unknown distribution on $\cX$, while in the fixed design setting $X_1,\ldots,X_n$ are some fixed points in $\cX$.

The Least Squares Estimator, or ERM with respect to square loss, is defined as
\begin{equation}\label{Eq:LSE}
    \hat{f}_n =   \Psi\left(\argmin_{f \in\cF}~ \sum_{i=1}^n (Y_i-f(X_i))^2\right),
\end{equation}
where $\Psi$ is a function that selects a particular solution in the set of possible minimizers (for example, a minimal norm solution).

One of the most important questions regarding ERM is its statistical performance as compared to other estimators, defined as maps from $\{(X_i,Y_i)\}_{i=1}^{n}$ to $\cF$ (or to $\reals^\cX$ for \textit{improper} methods). While there are multiple ways of measuring the statistical performance, perhaps the most popular is the \emph{minimax risk} \citep{tysbakovnon}, defined in the random design case for any estimator $\bar{f}_n$ as

$$
    \cR(\bar{f}_n, \cF,\PP) := \sup_{f_0\in\cF} \E_{x,\xi} \int (\bar{f}_n((X_1,Y_1),\ldots,(X_n,Y_n)) - f_0)^2d\PP,
$$
where $\E_{x,\xi}$ denotes expectation over the training data and the integral represents the expected out-of-sample performance with respect to $\PP$. One can also write this measure of performance as excess square loss
$$\sup_{f_0\in\cF} \E_{x,\xi} \E_{(X,Y)} (\bar{f}_n(X)-Y)^2 - \E_{(X,Y)}(f_0(X)-Y)^2.$$
We say that the ERM $\hat{f}_n$ is \emph{minimax optimal}, if for for all $n \geq 0$, 
\[
    \cR(\hat{f}_n, \cF,\PP) \lesssim \inf_{\bar{f}_n} \cR(\bar{f}_n, \cF,\PP),
\]
where $\lesssim$ denotes less or equal up to a constant that only depends on $\PP,\cF$.  The quantity $\inf_{\bar{f}_n} \cR(\bar{f}_n, \cF,\PP)$ is known as the \emph{minimax rate} for $(\cF,\PP)$. In the fixed design setting, the risk measure is defined in an analogous way, except that instead of drawing $n$ i.i.d. points from $\PP$, we consider a sequence of measures that are supported uniformly on  $n$ points. 

Clearly, the definitions of the risk and the minimax optimality measure ``the worst case scenario"  of a given estimator, and may hide the true statistical performance of the ERM in real-life applications (cf. \citep{bellec2017optimistic}). For example, as mentioned in the introduction, if $f_0$ is known to belong to a smaller class $\cH$, the relevant quantity is
\[
    \cR_{\cH}(\hat{f}_n,\cF,\PP):= \sup_{f_0 \in \cH}\E\int (\hat{f}_n - f_0)^2d\PP,
\]
 which may be significantly smaller than the minimax risk. We remark that the ERM $\hat{f}_n$ is still defined over $\cF$, due to computational or other considerations.
As an example, consider linear regression in $\R^d$ when the true coefficient vector is sparse, i.e. supported on $ k \ll d$ coordinates. Then, due to computational considerations, it is standard to replace the original problem of minimizing square loss over sparse vectors in $\R^d$ by minimization over a larger $\ell_1$ ball in $\R^d$ (the Lasso procedure).

The second example was already briefly mentioned in the introduction, and we expand on it here. Let $\cF_d$ be the family of convex $1$-Lipschitz functions on $\cX=[0,1]^d$, and let $\PP= \mathrm{Unif}(\cX)$. The subset $\cH_d$ (of `simple functions') is the set of $1$-Lipschitz $k$-affine piece-wise linear functions with $k=\Theta(1)$. It is well known that ERM over $\cH_d$ is NP-hard since the problem is highly non-convex; moreover, even estimating the number of pieces is computationally hard (cf. the recent paper \cite{ghosh2019max} for more details). In contrast, ERM over $\cF_d$ can be efficiently computed  \citep{ghosh2019max}. While the minimax rate for $(\cF_d,\PP_d)$ is $\Theta(n^{-4/(d+4)})$ \citep{dudley1999uniform,bronshtein1976varepsilon}, it was proved recently in \citep{kur2020suboptimality} that the risk of ERM is $\Tilde{\Theta}_d(\max\{n^{-2/d},n^{-4/(d+4)}\})$, which is minimax-suboptimal when $d\geq 5$. Furthermore, it was shown in \citep{han2016multivariate,feng2018adaptation} that
\begin{equation}\label{Eq:1}
    \cR_{\cH_d}(\hat{f}_n,\cF_d,\PP_d) \lesssim  \Tilde{O}(\max\{n^{-4/d},n^{-1}\}),
\end{equation}
which is \emph{significantly} smaller than both the risk of ERM and the minimax rate. When the ERM (or MLE) satisfies such improved bounds, we say that it exhibits \emph{adaptation} (cf. \citep{feng2018adaptation,kim2018adaptation,samworth2018recent,han2019isotonic,kur2020suboptimality}).

In this paper we answer the two following questions: Does there exist a \emph{uniform} lower bound on the \emph{minimal error} 
\[
    \inf_{f_0 \in \cF} \E_{x,\xi} \int (\hat{f}_n - f_0)^2d\QQ 
\]
of ERM $\hat{f}_n$, 
where $\QQ$ is either fixed or random design measure? Does the richness of the entire class $\cF$ affect the minimal error, or is there a more refined notion of complexity that governs its behavior?

\section{Main Results}

We start with definitions. For $n$ points $\mathbf{x}_n:= \{x_1,\ldots,x_n\}$ in $\cX$ and $\cG \subseteq \cF$, we define the  Gaussian averages of $\cG$ as
\[
    \EmpGaussAvgs(\cG):= \E_{\xi}\sup_{f\in \cG} \frac{1}{n} \sum_{i=1}^{n} \xi_i f(x_i),~~~~ 
    \GaussAvgs(\cG):= \E \EmpGaussAvgs(\cG).
\]
For a measure $\QQ$ on $\cX$ and $f:\cX \to \R$ we denote by $\| f\|_{\QQ}$ the $L_2(\QQ)$ norm of $f$. Finally, for any $\QQ$, $f \in \cF$, and $ r \geq 0$ we denote by $B_{\QQ}(f,r):= \{g \in \cF:\|g-f\|_{\QQ} \leq r \}$, the intersection of the $L_2(\QQ)$ ball around $f$ and the class $\cF$.

\subsection{Fixed Design}
We now state our sharp lower bound for the fixed design error, for simplicity of exposition under the assumption of uniform boundedness of $\cF$ (the general statement is given below in Lemma~\ref{Thm:fixed}).
\begin{corollary}
\label{cor:fixed_design}
    Let $\PP_n$ be the empirical measure on some $n$ points in $\cX$, and assume  $\cF\subseteq [-1,1]^\cX$ is convex. Then the minimal error of ERM over $\cF$ satisfies
\[
    \inf_{f_0 \in \cF} \E_{\xi}\int (\hat f_n - f_0)^2 d\PP_n \geq  64^{-1}(\EmpGaussAvgs(\cF)-Cn^{-1})^2,
\]
where $C$ is some positive absolute constant.
\end{corollary}

When $\cF$ is uniformly bounded (say, by $1$), a classical result in non-parametric statistics \citep{van2000empirical} and our theorem imply that
\[
    64^{-1}(\EmpGaussAvgs(\cF)-Cn^{-1})^2 \leq \inf_{f_0 \in \cF} \E_{\xi}\int (\hat f_n - f_0)^2 d\PP_n \leq  \underbrace{\sup_{f_0 \in \cF} \E_{\xi}\int (\hat f_n - f_0)^2 d\PP_n}_{=\cR(\hat{f}_n,\cF,\PP_n)}   \leq 2\EmpGaussAvgs(\cF).
\]
 Moreover, both of these bounds are tight, in the sense that they can be attained on certain families of functions, up to constants (cf. \cite{birge1998minimum,han2019isotonic}). 
Therefore, we conclude that in the \emph{fixed design} case, both the minimax risk and the minimal error of the ERM depend on the entire Gaussian complexity of $\cF$ (when it is convex and uniformly bounded). In particular, for the case of convex regression, Corollary~\ref{cor:fixed_design} recovers the rate in \eqref{Eq:1} (up to logarithmic factors) for the fixed design case, since with high probability the global complexity $\EmpGaussAvgs(\cF)$ is of the order $\max\{n^{-2/d}, n^{-1/2}\}$.

\subsection{Random Design}

We now turn to the random design setting, which is significantly more subtle. Before stating the result, we describe a direct proof strategy that fails. This approach would attempt to pass from the fixed design lower bound to the random design lower bound by relating the population and empirical norms $\norm{\cdot}_{\PP}$ and $\norm{\cdot}_{\PP_n}$, uniformly over the class. A statement of this type (which may be called ``upper isometry,'' in contrast with ``lower isometry'' studied, for instance, in \cite{mendelson2014learning}) could be derived under additional  assumptions on the geometry of $(\cF,\PP$), such as a small-ball condition \citep{mendelson2014learning}, Kolchinskii-Pollard entropy \citep{rakhlin2017empirical}, or an $\eps$-covering with respect to the $sup$-norm \cite{van2000empirical}. To the best of our knowledge, such upper-isometry statements can at best read
\[
    \|f - g\|^2_{\PP} \geq \frac{1}{2}\|f - g\|_{\PP_n}^2 - C \cdot\GaussAvgs(\cF)^2 \quad \forall f,g \in \cF,
\]
    where $C \geq 1$. Since $\GaussAvgs(\cF)^2$ is larger than the lower bound on the fixed design error, this technique does not appear to work.

    Moreover, a uniform lower bound of order $\GaussAvgs(\cF)^2$ in random design cannot be true in general. For instance, it was shown in a string of recent works \citep{liang2020multiple,belkin2019does,bartlett2020benign,tsigler2020benign} that it is possible to completely interpolate $Y_1,\ldots,Y_n$ (i.e. achieve zero empirical error) and still have a small generalization error (of order $n^{-c}$, for some $c \in (0,1)$), and even be minimax optimal (with an appropriate
    function $\Psi$ in Eq. \eqref{Eq:LSE}) . In these examples, because of the ability to interpolate any data, we know that $\GaussAvgs(B_n(f_0,1)) = \Theta(1)$; therefore, the lower bound in the fixed design case cannot be always true in random design.
    
    The last paragraph motivates the need to consider additional properties of the model $\cF$ and the underlying distribution $\PP$. With the interpolation examples in mind, we might hope that the relation between the global complexity of the class and complexity of local neighborhoods around the regression function $f_0$ may play a role in determining rates of convergence of ERM. To this end, for every $n$ and $f_0 \in \cF$, we define the following notion of complexity:
\begin{equation}\label{Eq:2}
    t_{n,\PP}(f_0,\cF):= \max\{t \in \R^{+}:\GaussAvgs(B_{\PP}(f_0,t)) \leq l_{\xi}\GaussAvgs(B_n(f_0,1))\},
\end{equation}
where $l_{\xi} \in (0,1)$ is a small absolute constant that will be chosen in the proofs. We remark that under the additional assumption of $\cF$ being uniformly bounded by $1$, we have that $\GaussAvgs(B_n(f_0,1)) \leq  \frac{1}{2}\GaussAvgs(B_n(f_0,2)) =  \frac{1}{2}\GaussAvgsF$, and thus we can replace the term on the right-hand side of \eqref{Eq:2} with global Gaussian averages $\GaussAvgsF$.

The quantity $t_{n,\PP}(f_0,\cF)$ is the maximal radius of the population ball around $f_0$ that has Gaussian complexity comparable to that of the entire class (in the uniformly bounded case), up to some absolute constant, or to a ball of constant radius within the class. As we show next, this local richness is necessary in order to avoid the rate being dominated by the global complexity of $\cF$. In the aforementioned interpolation examples we have both $ t_{n,\PP}(f_0,\cF) = O(n^{-c})$  and $\GaussAvgs(B_n(f_0,1)) = \Theta(1)$.  
 The last two relations must be true for any  $f_0 \in \cF$ for which ERM attains perfect fit to data, and yet a small generalization error of order $n^{-c}$.

We now state the main result of this paper for the random design setting, under the additional assumption of $\cF$ being uniformly bounded. Remarkably, $t_{n,\PP}(f_0,\cF)$ is the only additional quantity that we need to consider for a general uniform lower bound on a general family $\cF$. Specifically, we prove the following: 
\begin{theorem}\label{Thm:URandomDesign}
Let $\cF$ be a convex class of functions\footnote{We assume that $\cF$ is non-degenerate and contains at least two functions such that $\|f_1 - f_2\|_{\PP} \geq 1/2$.}  uniformly  bounded by one. Then for large enough $n$, the minimal error of ERM over $\cF$ is lower bounded as
\[
    \inf_{f_0 \in \cF}~ \frac{\E_{x,\xi}\int(\hat{f}_n - f_0)^2 d\PP}{\min\{\GaussAvgs(\cF)^2,t_{n,\PP}(f_0,\cF)^2\}} \geq c,
\]
where $c \in (0,1)$.
\end{theorem}

\begin{remark}
Notably, Theorem~\ref{Thm:URandomDesign} holds under only convexity and uniform boundedness assumptions on the class $\cF$.  
Furthermore, one can easily design  a convex uniformly bounded family and an $f_0 \in \cF$ such that the ERM attains an error of order $t_{n,\PP}(f_0,\cF)^2 \ll \GaussAvgsF^2$ for all $n$ that are large enough (for completeness see Section \ref{Sec:SimpleExample}). Therefore, under no additional assumption on $\cF$, the above lower bound is sharp up to absolute constants.
\end{remark}

An almost immediate corollary of this theorem is the following key insight on the behavior of the ERM procedure in the random design setting:

\begin{corollary}
Let $\cF$ be convex and uniformly bounded by $1$. For any $f_0 \in \mathcal{\cF}$ such that $$ \underbrace{\E_{x,\xi} \int (\hat{f}_n - f_0)^2d\PP}_{:=E^2(f_0)} \ll \GaussAvgsF^2,$$
there \emph{must} exists some constant $t(f_0) \leq c_1 \cdot E(f_0) $ such that 
\[
    \GaussAvgs(B_{\PP}(f_0,t(f_0))) = \Theta(\GaussAvgsF),
\]
where $c_1 \in (0,1)$ is some absolute constant.
\end{corollary}
Informally speaking, if ERM learns some $f_0 \in \cF$ at a rate faster than $ \GaussAvgsF^2$, then the local complexity of a population ball centered at $f_0$ with a very small radius must be as rich as the entire complexity of $\cF$. A more prescriptive recipe for guaranteeing such fast rates is an interesting direction of further work.

\subsection{Donsker and non-Donsker Classes}

The lower bounds stated thus far assumed little about the geometry of the class $\cF$ beyond convexity and global and local Gaussian averages. Under additional assumptions on the behavior of entropy $\log \cN(\eps,\cF,\PP)$ (defined as the logarithm of the smallest number of balls with respect to $L_2(\PP)$ of radius $\eps$ sufficient to cover $\cF$) or entropy with bracketing $\log \cN_{[]}(\eps,\cF,\PP)$ (defined as the logarithm of the smallest number of brackets $l_i,u_i \in \cF$ such that $l_i\leq u_i$, $\norm{l_i-u_i}_{\PP}\leq \eps$ and $\cF$ is contained in the union of the brackets), we can provide specific upper bounds on the Gaussian averages via chaining and other techniques. In particular, we say that a convex uniformly bounded $\cF$ is $\PP$-Donsker if $\log \cN_{[]}(\eps,\cF,\PP) \sim \eps^{-\alpha}$ with $\alpha \in (0,2)$ or if $\cF$ is parametric with $\log \cN_{[]}(\eps,\cF,\PP) \sim v\log (1/\eps)$ for some `dimension' $v$. In seminal works of \citep{birge1993rates} it was shown that for any $\PP$-Donsker class, the ERM is minimax optimal, i.e.  
$$
    \cR(\hat{f}_n,\cF,\PP)  \sim \inf_{\bar{f}_n}\cR(\bar{f}_n,\cF,\PP) \sim n^{-\frac{2}{2+\alpha}}.
$$ Note that for $\alpha \in (0,2)$ we have that $\GaussAvgsF \sim n^{-1/2}$.

The next result shows that without further assumptions we cannot learn \textit{any} function in a convex uniformly bounded $\PP$-Donsker class faster than a parametric rate. 
\begin{corollary}\label{Cor:Donsker}
    Let $\cF$ be a convex uniformly bounded $\PP$-Donsker class, and  let $X_1,\ldots,X_n \sim \PP$. Then
    \[
         n^{-1} \lesssim \inf_{f_0 \in \cF}\E\int (\hat{f}_n -f_0)^2 d\PP_n \sim \inf_{f_0 \in \cF}\E\int (\hat{f}_n -f_0)^2 d\PP  
    \]
\end{corollary}
This lower bound is sharp, namely there are classical $\PP$-Donsker classes, such as the convex regression example mentioned in the introduction and Section~\ref{sec:formal}, where ERM can attain a parametric rate (up to logarithmic factors) when optimizing over all convex Lipschitz functions, but only for $d \leq 4$ which puts us in the Donsker regime.

For non-Donsker classes, i.e when $\alpha > 2$, the ERM procedure may not be optimal. One can show that 
$$n^{-\frac{2}{2+\alpha}}\lesssim \cR(\hat{f}_n,\cF,\PP) \lesssim n^{-\frac{1}{\alpha}}$$
and both of these bounds can be tight, up to logarithmic factors. Furthermore, one can show that $$n^{-\frac{2}{2+\alpha}}\lesssim \GaussAvgsF \lesssim n^{-\frac{1}{\alpha}}$$
and, again, both of these can be tight. Our next corollary shows that in this regime, the fixed-design error is at least of the order $\GaussAvgsF^2$, i.e. it is impossible to learn at a parametric rate in the non-Donsker regime.  
\begin{corollary}\label{Cor:NonDonsker}
Let $\cF$ be a convex uniformly bounded non-$\PP$-Donsker class, and  let $X_1,\ldots,X_n \sim \PP$. Then the following holds:
    \[
         n^{-\frac{4}{2+\alpha}} \lesssim \GaussAvgsF^2 \lesssim \inf_{f_0 \in \cF}\E\int (\hat{f}_n -f_0)^2 d\PP_n 
    \]
\end{corollary}
The proof of these two corollaries appears in the appendix.
\begin{remark}
   Due to the geometry of general non-Donsker classes, in random design case the same lower bound may not hold. However, in all the examples in the  literature \citep{han2016multivariate,feng2018adaptation,kim2018adaptation,han2019isotonic,kur2020suboptimality} that study the adaptivity of ERM in non-Donsker families (such as convex functions when $d\geq 5$, isotonic functions when $d\geq 3$), the term of $t_{n,\PP}(f_0,\cF)$  of  Theorem   \ref{Thm:URandomDesign} is significantly larger than $\GaussAvgsF$. As a consequence, one may use Theorem   \ref{Thm:URandomDesign} to show that the bound in Eq. \eqref{Eq:1} is tight up to logarithmic factors.
\end{remark}

\subsection{General Lower Bound for Fixed Design}

In this section, we state the general lower bound for fixed design. In comparison to its consequence, Corollary~\ref{cor:fixed_design}, the version below captures complexity of local neighborhoods around regression functions that are close to $f_0$. Note that this lemma holds for any convex family (and not necessarily bounded). 
\begin{lemma}\label{Thm:fixed}
    Let $\cF$ be a convex family of functions and and let $x_1,\ldots,x_n \in \mathcal{X}$ be some $n$ points, and let $\PP_n:= n^{-1}\sum_{i=1}^{n}\delta_{x_i}$. For all $f_0 \in \cF$ define
    \begin{align}
        \label{eq:def_r0}
       r(f_0) := \argmax_{r \geq 0}\EmpGaussAvgs(B_n(f_0,r)) - \frac{r^2}{2}
    \end{align}
    and
    \[
        L_{x}(f_0):= \max_{g \in B_n(f_0,1) ,t \geq 0} \ \  \frac{\EmpGaussAvgs(B_n(g,t)) -\EmpGaussAvgs(B_n(f_0,r(f_0))-Cn^{-1}}{\|g - f_0 \|_{\PP_n}+t}
    \]
    where $C \in (1,\infty)$ is some absolute constant.
     Then the following lower bound holds:
    \[
\E_{\xi} \int (\hat{f}_n - f_0)^2 d\PP_n \geq \max\{\frac{(\EmpGaussAvgs(B_n(f_0,1))-Cn^{-1})^2}{4},L_{x}(f_0)^2\}. 
    \]
\end{lemma}
Note that Corollary \ref{cor:fixed_design} follows almost immediately from the last lemma. To see this, the convexity of $\cF$, and the uniform bounded by $1$ assumption imply that $$ \cW(\cF)=\EmpGaussAvgs(B_n(f_0,2))\leq  2\EmpGaussAvgs(B_n(f_0,1)).$$
\begin{remark}
   The second term in our lower bound may be significantly larger than $\EmpGaussAvgs(\cF)^2$. For example, the second term may be equal to $\EmpGaussAvgs(\cF)$ in several non-Donsker families that appear in \citep{birge1993rates,kur2020convex,birge2006model}. We also remark that  constant $\frac{1}{4}$ is tight (up to $o_n(1)$).
\end{remark}

The rest of this paper is devoted to proofs. While the fixed design lower bound follows a rather simple argument, the corresponding lower bound in the random design case is more subtle. In particular, we employ a particular version of Talagrand's inequality that, in our particular regime, provides control on certain empirical processes, while the more commonly used versions (including Bousquet's inequality) result in vacuous estimates.

\section{Proof of Lemma \ref{Thm:fixed}}
\paragraph{Notation}
Throughout this section, $c,c_1,c_2 \in (0,1)$ and $C,C_1,C_2 \in (1,\infty)$ are some absolute constants that may change from to line to line. Also $S_1,s_1,S_2,s_2$ are absolute constants, but we use this notation to emphasize that we have some freedom to control their size. We also use the notation $C(c_1,C_2)$ to mean that the constant depends on $c_1,C_2$.

To recap, we assume that  $\cF$ is a convex family of functions, $Y_i = f_0(x_i) + \xi_i$, where $\xi_i \sim N(0,1)$ i.i.d., $x_1,\ldots,x_n \in \cX$, and $f_0 \in \cF$. We write $\|f\|_n = \|f\|_{\PP_n}$ and $\inner{f,g}_n = \int fg d\PP_n$. With slight abuse of notation, we write $\inner{\bxi,f}_n = \frac{1}{n}\sum_{i=1}^n \xi_i f(x_i)$ for $\vec{\xi}:=(\xi_1,\ldots,\xi_n)$. We also abbreviate $B_n(f_0,t):=B_{\PP_n}(f_0,t)$ to be the $L_2(\PP_n)$ ball with respect to empirical measure $\PP_n$.

Recall the definition of $r(f_0)$ in \eqref{eq:def_r0}. The following lemma that was proven in \citep{chatterjee2014new}:
\begin{lemma}\label{Lem:Chat}[\citep[Thm 1.1]{chatterjee2014new}]
The following holds under the above assumptions: 
\begin{equation}
    \Pr\left(|\|\hat{f}_n - f_0\|_{n} - r(f_0)| \geq t\right) \leq
    \begin{cases}
        3\exp(-\frac{nt^2}{64}) &  t \geq r(f_0)\\
        3\exp(-\frac{nt^{4}}{64r(f_0)^2}) &0 \leq  t \leq r(f_0)
    \end{cases}
\end{equation}
Moreover, for each $t\geq 0 $ the following holds 
\begin{equation}\label{Eq:Widthconcentration}
    \Pr\left(|\inner{\hat{f}_n-f_0, \bxi}_n - \EmpGaussAvgs(B_n(f_0,r(f_0)))| \geq t\cdot r(f_0) \right)\leq
    \begin{cases}
        3\exp(-\frac{nt^2}{64}) & t \geq r(f_0)\\
        3\exp(-\frac{nt^{4}}{64r(f_0)^2}) &0 \leq  t \leq r(f_0)
    \end{cases}
\end{equation}
\end{lemma} 
Also, we state a simple corollary that follows from this lemma (cf. \citep{boucheron2013concentration},\citep[Thm 1.2]{chatterjee2014new})
\begin{corollary}\label{Cor:integrate}
The following two bounds hold 
\[
    \E \left|\inner{\hat{f}_n-f_0, \bxi}_n - \EmpGaussAvgs(B_n(f_0,r(f_0))) \right| \leq C_1\max\{r(f_0)^{3/2}n^{-1/4},n^{-1}\}.
\]
and 
\[
    \E \left|\|\hat{f}_n - f_0\|_{n}^2 - r(f_0)^2\right|  \leq C_2\max\{r(f_0)^{3/2}n^{-1/4},n^{-1}\}.
\]
\end{corollary}

\begin{proof}[Proof of Lemma \ref{Thm:fixed}]

For brevity, denote $\widehat{r}:= r(f_0)$, where $r(f_0)$ is defined in Lemma \ref{Lem:Chat}. 
Define
$$ 
    g_{\xi}:= \argmax_{h \in B_n(g,t)} \inner{h-g, \vec{\xi}}_n.
$$
Optimality of $\hat{f}_n$ and convexity of $\cF$ imply that $\inner{\nabla_f \norm{f-\by}^2_n|_{f=\hat{f}_n}, g-\hat{f}_n}_n \geq 0$ for any $g\in\cF$. In particular, for $g=g_\xi$ this implies
$$0\geq \E \inner{\bxi+f_0-\hat{f}_n, g_\xi - \hat{f}_n}_n,$$
where the expectation is over $\bxi$, conditionally on $x_1,\ldots,x_n$. For any $g\in \cF$, we may write the right-hand side as
\begin{align*}
    &\E \inner{\bxi+f_0-\hat{f}_n, g_\xi - g+ g - f_0 + f_0 - \hat{f}_n}_n\\
    & =  \EmpGaussAvgs(B_n(g,t)) -  
    \E\left[\inner{\bxi, \hat{f}_n-f_0}_n+ \inner{\hat{f}_n-f_0, g_{\xi} - g}_n+\inner{\hat{f}_n-f_0, g - f_0}_n\right] +\E \|\hat{f_n}-f_0\|_{n}^2
\end{align*}
where we used the definition of $g_{\xi}$ and the fact that $\E\inner{\bxi,g - f_0}_n=0$. Using  Corollary \ref{Cor:integrate}, we obtain a further lower bound of
\begin{align}
    \label{eq:expanded_lower_bd_inner}
    &\EmpGaussAvgs(B_n(g,t)) -   \EmpGaussAvgs(B_n(f_0,\widehat{r}))-\E\left[\inner{\hat{f}_n-f_0, g_{\xi} - g}_n+\inner{\hat{f}_n-f_0, g - f_0}_n\right] \notag\\& \quad + \widehat{r}^2 -C\widehat{r}^{3/2}n^{-1/4} - Cn^{-1} \notag \\
    & \geq \EmpGaussAvgs(B_n(g,t)) - \EmpGaussAvgs(B_n(f_0,\widehat{r}))
   -\E\left[\inner{\hat{f}_n-f_0, g_{\xi} - g}_n+\inner{\hat{f}_n-f_0, g - f_0}_n\right] \\
   &\quad + \widehat{r}^2/2 -C_1n^{-1}. \notag
\end{align}
To verify the last inequality, observe that $\widehat{r}^2/2 \geq C_1\widehat{r}^{3/2}n^{-1/4}$ when $\widehat{r} \geq C_2n^{-1/2}$ for $C_2$ that is large enough; on the other hand, if $\widehat{r} \leq C_2n^{-1/2}$, the $Cn^{-1}$ term is dominant for $C$ large enough. Since $\inner{\hat{f}_n-f_0, g - f_0}_n \leq \norm{\hat{f}_n-f_0}_n \norm{g-f_0}_n$ and $\inner{\hat{f}_n-f_0, g_{\xi} - g}_n\leq t\cdot \norm{\hat{f}_n-f_0}_n $,
we conclude that
\begin{equation}\label{Eq:newbasic}
\begin{aligned}
    0  \geq \EmpGaussAvgs(B_n(g,t)) - \EmpGaussAvgs(B_n(f_0,\widehat{r}))-\E[\|\hat{f}_n-f_0\|_{\PP_n}](\|g-f_0\|_{\PP_n}+t)-Cn^{-1}.
\end{aligned}
\end{equation}
By re-arranging the terms and using Jensen's inequality, we have
\[
    \E\| \hat{f_n} - f_0\|_{n}^2
    \geq  \left(\E\| \hat{f_n} - f_0\|_{n}\right)^2 \geq \left( \frac{\EmpGaussAvgs(B_n(g,t))- \EmpGaussAvgs(B_n(f_0,\widehat{r}))-Cn^{-1}}{t + \|f_0-g \|_{n}} \right)^2_+
\]
where $(a)_+^2 = \max\{a,0\}^2$. Since $L_x(f_0)$ in the statement of the Lemma is non-negative, the lower bound of $L_x(f_0)^2$ follows.

Now, for the first part of the lower bound, we have to consider two cases. The first one is when  $\EmpGaussAvgs(B_n(f_0,\widehat{r})) \geq 2^{-1}\EmpGaussAvgs(B_n(f_0,1))+\widehat{r}^2/2$, and we have
\begin{align*}
    \E\|\hat{f}_n - f_0\|_{n} &\geq \E \|\vec\xi\|_{n} \cdot \E \|\hat{f}_n - f_0\|_{n} \geq \E \inner{\vec{\xi}, \hat{f}_n - f_0}_n \\&\geq 2^{-1}\EmpGaussAvgs(B_n(f_0,1)) +\widehat{r}^2/2 -C\widehat{r}^{3/2}n^{-1/4} \geq 2^{-1}\EmpGaussAvgs(B_n(f_0,1)) -C_1n^{-1},
\end{align*}
where we used Cauchy-Schwartz inequality and Corollary \ref{Cor:integrate}.
In the other case, we use Eq. \eqref{Eq:newbasic} with 
$g=f_0$ and $t=1$:
\[
    \E \|\hat{f}_n - f_0 \|_{n} \geq \EmpGaussAvgs(B_n(f_0,1)) - \EmpGaussAvgs(B_n(f_0,\widehat{r})) -Cn^{-1}\geq 2^{-1}\EmpGaussAvgs(B_n(f_0,1))-Cn^{-1},
\]
concluding the proof.
\end{proof}

\section{Proof of Theorem \ref{Thm:URandomDesign}}
Throughout the proof of Theorem \ref{Thm:URandomDesign}, $\PP_n$ denotes the random empirical measure of $X_1,\ldots,X_n$. Denote by $\widehat{r} := \argmax \EmpGaussAvgs(B_n(f_0,r))-\frac{r^2}{2}$, with the hat emphasizing the dependence on $\Xn=(X_1,\ldots,X_n)$. We adopt the notation $\norm{\cdot}_n, \inner{\cdot,\cdot}_n$, $B_n$ in the previous section for the norm and the inner product with respect to $\PP_n$, and the $L_2(\PP_n)$ ball. Recall that we assumed that $\cF$ is not degenerate:  $\GaussAvgsF \geq c/\sqrt{n}$, for some $c \in (0,1)$.\footnote{See the proof of Lemma \ref{Lem:EventStand} for further details} 

\begin{proof}[Proof of Theorem \ref{Thm:URandomDesign}]
Denote 
\begin{equation}\label{Eq:tstar}
t_{*} := \min\{t_{n,\PP}(f_0,\cF), s_1\sqrt{\GaussAvgsF}, s_2\GaussAvgsF\}
\end{equation}
where $ s_1,s_2 \in (0,1)$ are small enough absolute constants that will be defined in the proof, and $ t_{n,\PP}(f_0,\cF)$ is defined in Eq. \eqref{Eq:2}.

Denote by $\cM$ the maximal separated set with respect to $L_2(\PP)$ at scale $6\sqrt{\GaussAvgsF}$, and let  \begin{equation}\label{Eq:cM}
M=\cM(6\sqrt{\GaussAvgsF},\cF,\PP)
\end{equation}
denote its size.  

For a constant $K_1 \in (1,\infty)$, let $\cE_1$ denote the high-probability event that is defined by the intersection of the events of Lemma \ref{Lem:Biso} and Lemma \ref{Lem:3}:
\begin{equation}
\begin{aligned}\label{Eq:E1}
\cE_1 := &\bigg\{\Xn: \sup_{f,g \in \cF}   \left|\|f-g\|_{n}^2 - \|f-g\|_{\PP}^2\right| \leq 10\GaussAvgs(\cF), 
\\& \sup_{h \in B_{\PP}(f_0,t_{*}),g \in \cM}\left| \inner{(g-f_0),(h-f_0)}_n-\E[(g-f_0)(h-f_0)]\right| \leq (8K_1)^{-1}\GaussAvgsF   \bigg\}.
\end{aligned}
\end{equation}
Further, define the events
\begin{equation}\label{Lem:Event}
\begin{aligned}
 \cE_2 &= \left\{\Xn: K_1^{-1}\GaussAvgsF \leq \EmpGaussAvgs(\cF) \leq K_1\GaussAvgsF  + Cn^{-1/2}\right\}, \\
 \cE_3 &= \left\{\Xn: \EmpGaussAvgs(B_{\PP}(f_0,t_{*})) \leq K_1\GaussAvgs(B_{\PP}(f_0,t_{*})) + C_1\GaussAvgsF^{1/2} n^{-1/2} \right\}, \\
    \cE &= \cE_1 \cap \cE_2 \cap \cE_3.
\end{aligned}
\end{equation}
Lemma \ref{Lem:EventStand}, proved in the appendix, shows that the event $\cE$ holds with probability of at least $0.9$. 
Note that under the event $\cE_1$, $\cM$ is also a $2\sqrt{\GaussAvgsF}$ separated set with respect to the random empirical measure $\PP_n$. Hence, we may apply Sudakov's minoration (Lemma \ref{Lem:SudMin}) with $\eps = 2\sqrt{\GaussAvgsF}$ and empirical measure $\PP_n$ defined on any $\Xn\in\cE$: \begin{align}
    \label{eq:packing_upper}
   c_1\sqrt{4\GaussAvgsF \cdot \frac{\log M}{n}} 
   \leq \EmpGaussAvgs(\cF)\leq K_1\GaussAvgsF + Cn^{-1/2}  \leq C_1K_1\GaussAvgsF,
\end{align}
where in the last inequality we used the assumption that $\GaussAvgsF \geq c\cdot n^{-1/2}$, and $C_1 \geq 0$ is defined to be large enough to satisfy the last inequality. Hence, the last equation implies that 
\begin{equation}\label{Eq:BoundOnM}
M \leq \exp(C_2 K_1^2 n\GaussAvgsF).
\end{equation}

First, recall the definition of $t_{n,\PP}(f_0,\cF)$
where in Lemma \ref{Lem:3} (that appears in the supplementary) we set $l_{\xi} = (256K_1)^{-3}$.
Recall Eq. \eqref{Eq:tstar}, where in Lemma \ref{Lem:3}  we set $s_1 = c(K,K_{1},C_2)$, and the three constants $K,K_1,C_2$ follow from Sudakov's minoration lemma, Talagrand's inequality, and Adamzcak's bound. We define $s_2:= 16^{-1}(K_1)^{-1}.$ 

Define the event 
\begin{align}
    \cA = \left\{(\bxi,\Xn): \hat{f}_n \in B_{\PP}(f_0,t_{*})\right\}
\end{align}
and, for any $\Xn$, define the conditional event
\begin{align}
    \cA(\Xn) = \left\{\bxi: \hat{f}_n \in B_{\PP}(f_0,t_{*})\right\}.
\end{align}
Assume by the way of contradiction that $\Pr_{x,\xi}(\cA) > 0.5$. Then, using the average principle (Fubini) and the fact $\Pr(\cE) \geq 0.9$, we can find an event $\cE_4 \subseteq \cE $  that has a probability of at least $0.4$ (when $n$ is large enough) such that 
\[
   \forall \Xn \in \cE_4  \quad \Pr_{\xi}(\cA(\Xn)) \geq 0.5.
\]
Our first step is to prove that, for all $\Xn \in \cE_4$,
\begin{equation}\label{Eq:abc}
    K_1\GaussAvgs(B_{\PP}(f_0,t_{*})) + l_{\xi}\GaussAvgsF \geq \EmpGaussAvgs(B_n(f_0,\widehat{r})).
\end{equation}
First, recall that $t_{*} \leq s_1\sqrt{\GaussAvgsF}$ and therefore under the event $\cE_1$, we have  
\begin{align}
    \label{eq:bd_on_emp_dist}\sup_{h \in B_{\PP}(f_0,t_{*})}\|h - f_0\|_{n}^2 \leq  11 \cdot \GaussAvgsF.
\end{align}
Now, for each $\Xn \in \cE_4 \subseteq \cE$, the map 
$\bxi\mapsto \sup_{h \in B_{\PP}(f_0,t_{*})} \inner{\vec{\xi},h-f_0}_n \ $
is Lipschitz with constant at most $$\sup_{h\in B_{\PP}(f_0,t_*)} n^{-1/2}\|h - f_0 \|_{n} \leq C_3\sqrt{\GaussAvgsF n^{-1}}$$
by \eqref{eq:bd_on_emp_dist},
and thus by Lipschitz concentration (Lemma \ref{Lem:Lipschtiz}), conditionally on $\Xn$,
\[
    \Pr_{\xi}\left(|\sup_{h \in B_{\PP}(f_0,t_{*})} \inner{\vec{\xi},h-f_0}_n - \EmpGaussAvgs(B_{\PP}(f_0,t_{*}))| \geq \eps \right) \leq 2\exp(-C n\GaussAvgsF^{-1}\eps^{2})
\]
for some absolute constant $C$. By setting $\eps = C_5(n^{-1}\GaussAvgsF)^{1/2}$ in the last equation, we may define the event
\[
   \cA_{1}(\Xn) =\left\{\bxi: |\sup_{h \in B_{\PP}(f_0,t_{*})}\inner{\vec{\xi},h -f_0}_n - \EmpGaussAvgs(B_{\PP}(f_0,t_{*}))| \leq C_5\sqrt{\GaussAvgsF n^{-1}} \right\} \cap \cA(\Xn).
\]
that holds with probability of at least $0.25$ (over $\xi$) for any $\Xn \in \cE_4$ .

Before defining the next event, observe that
\begin{align*}
\widehat{r} &= \argmax_{r \geq 0}\EmpGaussAvgs (B_n(f_0,r))-r^{2}/2 \leq   2\sqrt{\EmpGaussAvgs(\cF)},
\end{align*}
according to Lemma \ref{Lem:Chat} and the fact that for $r_n:=2\sqrt{\EmpGaussAvgs(\cF)}$ we have $\EmpGaussAvgs(B_n(f_0,r_n))-r_n^{2}/2 \leq 0$. As we already argued in \eqref{eq:packing_upper}, for any $\Xn \in \cE_2$ we have that $\EmpGaussAvgs(\cF) \leq  C_1K_1\GaussAvgsF$ for some absolute constant $C_1$, and thus 
\begin{align}
    \label{eq:bd_on_rhat}
    \forall \Xn\in\cE_2,~~~~ \widehat{r}\leq C\sqrt{K_1\GaussAvgs(\cF)}.
\end{align}
Now, from Eq. \eqref{Eq:Widthconcentration} in Lemma  \ref{Lem:Chat}, for $C_8$ large enough, the event 
\[
     \left\{\bxi: |\inner{\vec{\xi},\hat{f}_n -f_0}_n - \EmpGaussAvgs(B_n(f_0,\widehat{r}))| \leq C_8(n^{-1/4}\widehat{r}^{3/2} +n^{-1}) \right\} 
\]
holds with probability of at least $0.5$, and thus, in view of \eqref{eq:bd_on_rhat}, for all $\Xn \in \cE_4$, the event
\[
\cA_2(\Xn) =
     \left\{\bxi: |\inner{\vec{\xi},\hat{f}_n -f_0}_n - \EmpGaussAvgs (B_n(f_0,\widehat{r}))| \leq C_6K_1\GaussAvgs(\cF)^{3/4}n^{-1/4} \right\} \cap \cA_1(\Xn)
\]
that holds with probability of at least $0.1$ over $\xi$. 

We are now ready to prove Eq. \eqref{Eq:abc}, using the fact that $\cA_2(\Xn)$ is not empty for each $\Xn \in \cE_4$. To this end, fix $\Xn \in \cE_4$ and $\vec{\xi}\in\cA_2(\Xn)$. First, by definition of $\cE_3$, we have \begin{align*}
     K_1\GaussAvgs(B_{\PP}(f_0,t_{*})) \geq \EmpGaussAvgs(B_{\PP}(f_0,t_{*})) - C\sqrt{\GaussAvgsF n^{-1}}
\end{align*}
which can be further lower bounded, by definition of  $\cA_1(\Xn)$, by
\begin{align*}
     \sup_{h \in B_{\PP}(f_0,t_{*})} \inner{\bxi,h-f_0}_n - C\sqrt{\GaussAvgsF n^{-1}}. 
\end{align*}
Since $\vec\xi \in \cA_2(\Xn) \subseteq \cA(\Xn)$, the above expression is further lower bounded by
\begin{align*}
     &\inner{\bxi,\hat{f}_n-f_0}_n - C\sqrt{\GaussAvgsF n^{-1}}
\end{align*}
which, under the assumption of  $\vec\xi \in \cA_2(\Xn)$, is lower bounded by 
\begin{align*}
     &\EmpGaussAvgs(B_n(f_0,\widehat{r})) - C_2\sqrt{\GaussAvgsF n^{-1}}  - C_6K_1\GaussAvgs(\cF)^{3/4}n^{-1/4}.
\end{align*}
When $n$ is large enough, the above estimate is lower bounded by
\begin{equation*}
\begin{aligned}
     \EmpGaussAvgs(B_n(f_0,\widehat{r})) - l_{\xi}\GaussAvgsF.
\end{aligned}
\end{equation*}
To see this, observe that under the assumption of $\GaussAvgs(\cF) \geq c/\sqrt{n}$, both
$
    \sqrt{\GaussAvgs(\cF)n^{-1}}= o_n(\GaussAvgs(\cF) )
$
and 
$
    \GaussAvgs(\cF)^{3/4}n^{-1/4} = o_n(\GaussAvgs(\cF)).
$
Therefore, we proved Eq. \eqref{Eq:abc} holds, namely that
$$
    K_1\GaussAvgs(B_{\PP}(f_0,t_{*}))+l_{\xi}\GaussAvgsF \geq \EmpGaussAvgs(B_n(f_0,\widehat{r}))
$$
for all $\Xn \in \cE_4$.  Using the definition of $l_{\xi}= (256K_1)^{-3}$, we have
$$
    \GaussAvgs(B_{\PP}(f_0,t_{*})) \leq l_{\xi}\GaussAvgsF \leq 128^{-3}K_1^{-3}\GaussAvgsF,
$$ and thus for any $\Xn\in\cE_4$,
\begin{align}
\label{eq:compare_WF_and_What}
8^{-1}K_1^{-1}\GaussAvgsF \geq \EmpGaussAvgs(B_n(f_0,\widehat{r})).
\end{align}
By Lemma \ref{Lem:1} and \eqref{eq:compare_WF_and_What}, for any $\Xn \in \cE_4$, 
\begin{align*}
    0 \geq& (2^{-1}K_1^{-1}-8^{-1}K_{1}^{-1})\GaussAvgsF  -\E_{\xi}\max_{g \in \cM} \inner{g-f_0,\hat{f_n}-f_0}_n
  -16\sqrt{\GaussAvgsF}\widehat{r},
\end{align*}
and since $\Xn \in \cE_4 \subseteq \cE_1$, we also have 
\begin{align}
    \label{Eq:abcd}
    0 \geq& (2^{-1}K_1^{-1}-8^{-1}K_1^{-1}-8^{-1}K_1^{-1})\GaussAvgsF  -\sup_{h \in B_{\PP}(f_0,t_{*}),g \in \cM} \int(g-f_0)(h-f_0)d\PP
  -16\sqrt{\GaussAvgsF}\widehat{r} \notag
  \\& \geq (4K_{1})^{-1}\GaussAvgsF-\max_{g \in \cM}\|g-f_0\|t_{*}-16\sqrt{\GaussAvgsF}\widehat{r} \notag
  \\&\geq (4K_{1})^{-1}\GaussAvgsF-2t_{*}-16\sqrt{\GaussAvgsF}\widehat{r}
\end{align}
where we used the Cauchy-Schwartz inequality, the fact that $\cF \subset [-1,1]^{\cX}$, and the definition of $l_{\xi} = (256K_1)^{-3} $.

If $16\sqrt{\GaussAvgsF}\widehat{r} < (8K_1)^{-1}\GaussAvgsF$, then the last equation implies that
$$
    s_2\GaussAvgsF = (16K_1)^{-1}\GaussAvgsF < t_{*}.
$$ 
However, this inequality contradicts the definition of $t_{*}$, and thus cannot hold for any $\Xn \in \cE_4$. 
In the other case, we assume that $
   16\sqrt{\GaussAvgsF}\widehat{r}  \geq (8K_1)^{-1}\GaussAvgsF
$, or equivalently, $\widehat{r} \geq (128K_1)^{-1}\sqrt{\GaussAvgsF}$.
Now, from Lemma \ref{Lem:Chat} one can see that the maximizing value $\widehat{r}$ ensures
\[
    \EmpGaussAvgs(B_n(f_0,\widehat{r})) - 2^{-1}\widehat{r}^2 > 0 
\]
and hence
$$\EmpGaussAvgs(B_n(f_0,\widehat{r})) > 2^{-1}(128K_1)^{-2}\GaussAvgsF.$$
Therefore, under the event $\cE_4$ and by Eq. \eqref{Eq:abc}
\begin{align*}
 2K_{1}l_{\xi}\GaussAvgsF \geq K_1\GaussAvgs(B_{\PP}(f_0,t_{*})) + l_{\xi}\GaussAvgsF \geq \EmpGaussAvgs(B_n(f_0,\widehat{r})) > 2^{-1}(128K_1)^{-2}\GaussAvgsF.
\end{align*}
Once again, we have a contradiction for any $\Xn \in \cE_4$, since we assumed that $l_{\xi} =  (256K_1)^{-3}$. 

Therefore, we showed that Eq. \eqref{Eq:abcd}, cannot hold under the event $\cE_4$, i.e. the set $\cE_4$ is empty. This contradicts our earlier conclusion that $\Pr(\cE_4)\geq 0.4$, which was made under the assumption that event $\cA$ has probability at least $0.5$. Hence,  we conclude that $\Pr(\cA) \leq 0.5$, or, equivalently, with probability at least $0.5$, $\hat{f_n} \notin B_{\PP}(f_0,t_{*})$.  Therefore, we must have that 
\[
    \E\int(\hat{f}_n  -f_0)^2d\PP \geq \frac{t_{*}^2}{2} = \frac{1}{2} \min\{ t_{n,\PP}(f_0,\cF)^2,s_1^2\GaussAvgsF,s_2^2\GaussAvgsF^2\} \geq c_1 \cdot \min\{ t_{n,\PP}(f_0,\cF)^2,\GaussAvgsF^2\}. 
\]
where in the last inequality, we used the fact that $$\GaussAvgsF \leq \GaussAvgs([-1,1]^{\cX}) \leq   \E|\xi| \leq \sqrt{\E \xi^2} =1.$$ 
The theorem follows. 
\end{proof}

\section*{Acknowledgements} 

We acknowledge support from the NSF through award
DMS-2031883 and from the Simons Foundation through Award 814639
for the Collaboration on the
Theoretical Foundations of Deep Learning. We further acknowledge support from  NSF through grant DMS-1953181 and ONR through grants N00014-20-1-2336 and N00014-20-1-2394.

\bibliographystyle{alphabetic}
\bibliography{bib}

\newpage
\appendix

\section{Lemmas}
\begin{lemma}\label{Lem:1}
Under the event $\cE$ in \eqref{Lem:Event}, and for $n$ that is large enough, the following holds:
\begin{equation}\label{Eq:stationaryRandom}
\begin{aligned}
  0 \geq& 2^{-1}K_1^{-1}\GaussAvgsF - \EmpGaussAvgs(B_{n}(f_0,\widehat{r})) - \E_{\xi}\max_{g \in \cM} \inner{g-f_0,\hat{f_n}-f_0}_n-16\sqrt{\GaussAvgsF}\widehat{r},
\end{aligned}
\end{equation}
where $K_1$ is defined in Eq. \eqref{Lem:Event}, and set $\cM$ is defined in Eq. \eqref{Eq:cM}.
\end{lemma}
\begin{lemma}{\label{Lem:3}}
 Let $X_1,\ldots,X_n \underset{i.i.d}{\sim} \PP$, then the following holds with probability of at least $1-K\exp(-n\GaussAvgsF)$ 
\begin{align*}
 &\sup_{ g \in \cM, h \in B_{\PP}(f_0,t_{*})} |\inner{h-f_0,g-f_0}_n -\int_{\cX} (h-f_0)(g-f_0)d\PP| \leq (8K_1)^{-1}\GaussAvgsF.
\end{align*}
where $\cM$ is defined in Eq. \eqref{Eq:cM}, $t_{*}$ is defined in Eq. \eqref{Eq:tstar}, and $K_1,K$ are defined in Eq. \eqref{Lem:Event}, Lemma \ref{Lem:Talagrand}.
\end{lemma}
\begin{lemma}\label{Lem:EventStand}
The event $\cE$ defined in Eq. \eqref{Lem:Event} holds with probability of at least $0.9$.
\end{lemma}

\subsection{Auxiliary Lemmas}
\begin{lemma}\label{Lem:Biso}[\citep[pgs. 25-26]{koltchinskii2011oracle}]\label{Lem:Iso}
Let $\cF\subseteq [-1,1]^\cX$ be family of functions. Then with probability of at least $1-2\exp(-c_1n\GaussAvgsF)$,
\[
   \quad \forall f,g \in \cF  \quad  \left|\|f-g\|_{n}^2 - \|f-g\|_{\PP}^2\right| \leq 10\GaussAvgs(\cF),  
\]
and
\[
      \|\hat{f}_n-f_0\|_{n}^2  \leq 10\GaussAvgs(\cF).
\]
\end{lemma}

\begin{lemma}[Sudakov's minoration lemma]\label{Lem:SudMin}
     Let $\cH \subset [-1,1]^\cX$. There exists a constant $c_1$ such that for any $\PP_n$,
    \[
    c_1\sup_{\eps \geq 0}\eps\sqrt{\frac{\log \cM(\eps,\cH,\PP_n)}{n}} \leq \EmpGaussAvgs(\cH).
    \]
    where $\cM(\eps,\cH,\PP_n)$ denotes the size of the largest $\epsilon$-separated set in $\cH$ with respect to $L_2(\PP_n)$.
\end{lemma}
The next two lemmas appear in \citep[pgs. 24-25]{koltchinskii2011oracle}, \citep{adamczak2008tail}.
\begin{lemma}[Talagrand's inequality]\label{Lem:Talagrand}
Let $X_1,\ldots,X_n \underset{i.i.d.}{\sim} \PP$, and $\cH\subseteq [-U,U]^\cX$ be a family of functions. Let $Z=\sup_{f \in \cH}|n^{-1}\sum_{i=1}^{n}f(X_i)-\E[f]|$. Then there exists an absolute constant $K \geq 0$ such that for any $s \geq 0$
\begin{align*}
    \Pr \left( \left|Z- \E Z\right| \geq s \right) \leq K\exp\left(-K^{-1}U^{-1}\log(1+\frac{sU}{V^2})ns\right),
\end{align*}
where $ V^2 = \sup_{f \in \cH}\int f^2d\PP$.
\end{lemma}
\begin{lemma}[Adamczak's inequality]\label{Lem:Adam}
Let $\cG$ be a centred family of functions supported on $\cD$, and $\QQ$ be some distribution on $\cD$. Let $Z=\sup_{g \in \cG}|n^{-1}\sum_{i=1}^{n}g(X_i)|$. Assume that there exists an envelope function $G$ such that $|g(x)| \leq G(x)$ for all $g \in \cG, x \in \cD$. Then, the following holds for all $t \geq 0$
\begin{align*}
    K_{2}^{-1}\E Z
    -V\sqrt{\frac{t}{n}} -  \frac{\|\max_{1\leq i \leq n}f'(X_i)\|_{\psi_1}t}{n} \leq Z
    \leq K_{2}\E Z +V\sqrt{\frac{t}{n}} + \frac{\|\max_{1\leq i \leq n}f(X_i)\|_{\psi_1}t}{n},
\end{align*}
where $ V^2 := \sup_{g \in \cG}\int g^2d\QQ$, and $K_2 \in (1,\infty)$ is some universal constant, and $\psi_1$ is the Orlicz norm.
\end{lemma}
\begin{lemma}[Lipschitz Concentration]\label{Lem:Lipschtiz}
Let $\xi_1,\ldots,\xi_n \underset{i.i.d}{\sim} N(0,1)$, and $f:\mathbb{R}^n\to\mathbb{R}$ be a $L$-Lipschitz function with respect to $\norm{\cdot }_2$. Then, for all $\eps > 0$,
\[
    \Pr( |f - \E[f]| \geq \eps) \leq \exp(-c\eps^2L^{-2}).
\]

\end{lemma}

\section{Proofs}
\begin{proof}[Proof of Lemma \ref{Lem:1}]

We invoke the lower bound of Eq.~\eqref{eq:expanded_lower_bd_inner} with $g = f_0$ and $t=2$, implying
\begin{align}
    \label{eq:lem12_first}
         0 &\geq \EmpGaussAvgs(B_{n}(f_0,2)) - \EmpGaussAvgs(B_{n}(f_0,\widehat{r}))
   -\E \inner{\hat{f}_n-f_0,g_{\xi} - f_0}_n -Cn^{-1}  \notag\\
   &= \EmpGaussAvgs(\cF) - \EmpGaussAvgs(B_{n}(f_0,\widehat{r}))
   -\E \inner{\hat{f}_n-f_0,g_{\xi} -\Pi(g_{\xi})+\Pi(g_{\xi})-f_0}_n -Cn^{-1},
\end{align}
where $\Pi(g_{\xi}):= \argmin_{g \in \cM }\|g_{\xi}-g\|_{\PP_n}$, and the equality follows for the fact that for $\cF\subseteq [-1,1]^\cX$  we have $B_{n}(f_0,2) = \cF$.

Now, recall that $\cM$ is a maximal $6\sqrt{\GaussAvgsF}$-separated set with respect to  $L_2(\PP)$, and therefore also a $12\sqrt{\GaussAvgsF}$-net with respect to $L_2(\PP)$. Therefore, under the event $\cE$ it is also a  $16\sqrt{\GaussAvgsF}$-net with respect to $L_2(\PP_n)$, and, in particular, $\|\Pi(g_\xi)-g_\xi\|_{\PP_n} \leq  16\sqrt{\GaussAvgsF}.$
Hence, we can rewrite \eqref{eq:lem12_first} as
\begin{align*}
  &\E_{\xi} \max_{g \in \cM} \inner{\hat{f}_n-f_0,g-f_0}_n  \\
  &\geq  \EmpGaussAvgs(\cF)-\EmpGaussAvgs(B_{n}(f_0,\widehat{r}))-\E_{\xi}  \inner{\hat{f}_n-f_0,g_{\xi} -\Pi(g_{\xi})}_n -Cn^{-1}
  \\&\geq K_1^{-1}\GaussAvgsF-\EmpGaussAvgs(B_{n}(f_0,\widehat{r}))-16\sqrt{\GaussAvgsF}\E_{\xi}\|f-\hat{f}_n\|_{\PP_n} -Cn^{-1}
\end{align*}
Now, we proceed by using the first part of Corollary \ref{Cor:integrate} and the assumption of lying in $\cE$. The last expression is lower-bounded by
\begin{equation}\label{Eq:bb}
  K_1^{-1}\GaussAvgsF-\EmpGaussAvgs(B_{n}(f_0,\widehat{r}))-16\sqrt{\GaussAvgsF}(\widehat{r}+C\widehat{r}^{1/2}n^{-1/4}) .
\end{equation}
According to \eqref{eq:bd_on_rhat}, under the event $\cE$, we have $$\widehat{r} \leq C_3\sqrt{K_1\GaussAvgsF}$$
for some constant $C_3$. Thus the expression in Eq. \eqref{Eq:bb} is further lower-bounded by 
\begin{align*}
  & K_1^{-1}\GaussAvgsF-\EmpGaussAvgs(B_{n}(f_0,\widehat{r}))-16\sqrt{\GaussAvgsF}\widehat{r}-C_4\sqrt{K_1}\GaussAvgsF^{3/4}n^{-1/4}-Cn^{-1}
  \\&\geq (2K_1)^{-1}\GaussAvgsF-\EmpGaussAvgs(B_{n}(f_0,\widehat{r}))-16\sqrt{\GaussAvgsF}\widehat{r}
\end{align*}
where the last inequality holds when $n$ is large enough. To see this, recall that $\GaussAvgs(\cF) \geq c/\sqrt{n}$ and under this assumption both $ n^{-1}=o_n(\GaussAvgsF)$  and $\GaussAvgsF^{3/4}n^{-1/4}=o_n(\GaussAvgsF)$ hold. Therefore, the lemma follows.
\end{proof}
\begin{proof}[Proof of Lemma \ref{Lem:3}]
First, denote by $\|\PP_n - \PP\|_{\cH}:= \sup_{h \in \cH}|n^{-1}\sum_{i=1}^{n}h(X_i)-\E[h]|$, and for each $g_i \in \cM$, define $\cG_i=\{(h-f_0)(g_i-f_0): h \in B_{\PP}(f_0,t_{*})\}.$  By Talagrand's inequality (Lemma \ref{Lem:Talagrand}), the following holds for and $u \geq 0$
\begin{align*}
    \Pr\left(|\|\PP_n - \PP\|_{\cG_i} -\E\|\PP_n - \PP\|_{\cG_i}| \geq u \right) \leq K\exp\left(-nK^{-1}\log(1+ \frac{4^{-1}us_1^{-2}}{\GaussAvgsF})u \right)
\end{align*}
where we used the fact that $V^2 \leq \sup_{h \in \cF}\|g-f_0\|_{\infty}^2t_{*}^2 \leq 4s_1^{2}\GaussAvgsF$.
Now, we set $u =(16K_1)^{-1}\GaussAvgsF$ in the last equation
\begin{align*}
    &\Pr\left(|\|\PP_n - \PP\|_{\cG_i} -\E\|\PP_n - \PP\|_{\cG_i}| \geq (16K_1)^{-1}\GaussAvgsF \right) \\&\leq K\exp\left(-n(16K\cdot K_{1})^{-1}\GaussAvgsF \log(1+ 4^{-1}(16K_1)^{-1}s_{1}^{-2})\right).
\end{align*}

Next, we aim to take a union bound over $\cM$, and recall that $\log M \leq C_2K_1^2n\GaussAvgsF \leq C(K_1)n\GaussAvgsF$, for some absolute constant that does not depend on $s_1$. Therefore, we may choose \begin{equation}
s_1:=  c(K,K_1,C_2)
\end{equation}
where $c(K,K_1,C_2)$ is a constant that satisfies the following: 
\begin{equation*}
   \Pr\left(|\|\PP_n - \PP\|_{\cG_i} -\E\|\PP_n - \PP\|_{\cG_i}| \geq (16K_1)^{-1}\GaussAvgsF \right) \leq  K\exp(-2C_2K_1n\GaussAvgsF).
\end{equation*}
Therefore, we have
\begin{align*}
    \Pr\left(\exists  1 \leq i \leq M:|\|\PP_n - \PP\|_{\cG_i} -\E\|\PP_n - \PP\|_{\cG_i}| \geq (16K_1)^{-1}\GaussAvgsF \right) &\leq MK\exp(-2C_2K_1n\GaussAvgsF)) \\&\leq K\exp(-C_2K_1n\GaussAvgsF) \\&\leq K\exp(-n\GaussAvgsF).
\end{align*}
We conclude that with probability of at least $1-K\exp(-n\GaussAvgsF)$ the following holds for $\cG:= \{(h-f_0)(g-f_0) : g \in \cM, h \in B_{\PP}(f_0,t_{*})\}$:
\begin{equation}\label{Eq:b2}
    \|\PP_n -\PP\|_{\cG} \leq  \max_{1 \leq i \leq M}\E\|\PP_n -\PP\|_{\cG_i} + (16K_1)^{-1}\GaussAvgsF.
\end{equation}
The lemma will follow as soon as we  show that 
\[
        \max_{1 \leq i \leq M}\E\|\PP_n -\PP\|_{\cG_i} \leq (16K_1)^{-1}\GaussAvgsF.
\]
In order to prove the last inequality, we first apply the symmetrization lemma (cf. \citep[p. 20]{koltchinskii2011oracle}) and majorize the resulting Rademacher averages by a constant multiple of the Gaussian averages
\begin{equation}\label{Eq:b1}
    \E\|\PP_n -\PP\|_{\cG_i} \leq 8\GaussAvgs(\cG_i)
\end{equation}
where we used the fact that $0 \in \cG_i$.

Next, since $\|g_i-f_0\|_{\infty} \leq 2$, a standard argument (e.g. \cite[Theorem 3.1.17]{gine2016mathematical}) gives
\begin{align*}
     \E_{\xi} \sup_{h \in B_{\PP}(f_0,t_{*})} n^{-1}\sum_{k=1}^{n}(h-f_0)(g_i-f_0)(X_k)\xi_k &\leq  2\E_{\xi} \sup_{h \in B_{\PP}(f_0,t_{*})} n^{-1}\sum_{k=1}^{n}(h-f_0)(X_i)\xi_k.
\end{align*}
Then, by taking expectation over $X_1,\ldots,X_n$ over the last equation and by Eq. \eqref{Eq:b1}, we conclude  
\[
         \E\|\PP_n -\PP\|_{\cG_i} \leq 16\GaussAvgs(B_{\PP}(f_0,t_{*})) \leq 16l_{\xi}\GaussAvgsF\leq (16K_1)^{-1}\GaussAvgsF,
\]
where we set $l_{\xi} = (256K_1)^{-3}$. Then, by Eq. \eqref{Eq:b2} and  the last equation, the claim follows.
\end{proof}
\begin{proof}[Proof of Lemma \ref{Lem:EventStand}]
It is enough to show that $\cE_2$, $\cE_3$ hold with probability of at least $0.99$ for $n$ large enough.  First, we prove this claim for $\cE_2$.

We aim to apply Adamczak bound for concentration of the suprema of unbounded empirical processes (Lemma~\ref{Lem:Adam}). For this purpose, define the family of functions $\cG:= \{yf(x),y \in \R, f\in \cF-f_0\}$, and the distribution $\QQ = \PP \otimes N(0,1)$. Note that  $\cF \subseteq [-1,1]^{\cX}$ and, $\xi$ is Gaussian. Therefore, by Pisier's inequality (cf. \cite{pisier1983some},\citep[Eq. 13]{adamczak2008tail}), we have $$\|\max_{1\leq i \leq n}|\xi_if(X_i)|\|_{\psi_1} \leq C\log(n)\max_{1\leq i\leq n} \||\xi_if(X_i)|\|_{\psi_1} \leq  C_2\log(n).$$
By Adamczak's bound (Lemma \ref{Lem:Adam}), 
\begin{equation}\label{Eq:simEq}
\begin{aligned}
    &K_{2}^{-1}\E_{x,\xi} \sup_{f \in \cF-f_0} |\frac{1}{n}\sum_{i=1}^{n}f(X_i)\xi_i|  -\frac{10}{\sqrt{n}} - \frac{C\log(n)}{n} \\
    &\leq \sup_{f \in \cF-f_0}|\frac{1}{n}\sum_{i=1}^{n}f(X_i)\xi_i| \leq K_{2}\E_{x,\xi}\sup_{f \in \cF-f_0} |\frac{1}{n}\sum_{i=1}^{n}f(X_i)\xi_i|  + \frac{10}{\sqrt{n}} + \frac{C\log(n)}{n},
\end{aligned}
\end{equation}
with probability of at least $0.99$ both $X_1,\ldots,X_n$ and $\vec\xi$.

Now, using the average principle, for $n$ large enough,  we can find an event $\cE_{7}$ (that depends only on $X_1,\ldots,X_n$)  that holds with probability $0.98$, such that for any fixed $\Xn \in \cE_{7}$, there exists an event $\cA_3(\Xn)$  of probability at least $0.98$ (over $\vec\xi$) such that Eq. \eqref{Eq:simEq} holds. For each $\Xn \in \cE_7$, Lemma \ref{Lem:Lipschtiz} (with Lipschitz constant $\sup_{f \in \cF}\|f-f_0\|_{n} \leq 2$) implies that the middle term in \eqref{Eq:simEq} is, with high probability, within $C n^{-1/2}$ from its expectation (with respect to $\vec{\xi}$).
Therefore, we have for all $\Xn \in \cE_7$:
\begin{align*}
    &K_{2}^{-1}\E_{x,\xi} \sup_{f \in \cF-f_0} |n^{-1}\sum_{i=1}^{n}f(X_i)\xi_i|  -\frac{C}{\sqrt{n}}\\
    &\leq \E_{\xi}\sup_{f \in \cF-f_0}|n^{-1}\sum_{i=1}^{n}f(X_i)\xi_i| 
    \leq K_{2}\E_{x,\xi}\sup_{f \in \cF-f_0} |n^{-1}\sum_{i=1}^{n}f(X_i)\xi_i|  + \frac{C}{\sqrt{n}}. 
\end{align*}
Finally, since $ 0 \in \cF-f_0$, we have
\[
    \E_{\xi}\sup_{f \in \cF-f_0} n^{-1}\sum_{i=1}^{n}f(X_i)\xi_i \leq \E_{\xi}\sup_{f \in \cF-f_0} |n^{-1}\sum_{i=1}^{n}f(X_i)\xi_i| \leq 2\E_{\xi}\sup_{f \in \cF-f_0} n^{-1}\sum_{i=1}^{n}f(X_i)\xi_i.
\]
Hence, the last two equations imply that  when $\GaussAvgsF \geq C_1n^{-1/2}$, for $C_1$ that is large enough, the claim follows for $\cE_2$. To handle the remaining case of $\GaussAvgsF \leq C_1n^{-1/2}$, recall that we assumed that our class is not degenerate (i.e it has two functions that are $\|f_1-f_2\|_{\PP} \geq 0.5$. Then, it is easy to see that with probability of $0.99$ it holds that
\[
    \EmpGaussAvgs(\cF-f_0) \geq \GaussAvgs(\{0,f_2-f_0,f_1-f_0\}) \geq \E\max \{n^{-0.5}g,0\} \geq  c\cdot n^{-1/2} \geq c\cdot C_1^{-1}\GaussAvgsF,
\]
where $g \sim N(0,1/4)$. Therefore, for some $K_1^{-1} = c(K_2,c)$, the claim follows for $\cE_2$.

Next, we handle $\cE_3$. By using the definition of $B_{\PP}(f_0,t_{*})$,  and similar considerations that led to Eq. \eqref{Eq:simEq}, we have
\begin{equation}\label{Eq:simEq2}
\begin{aligned}
    \sup_{f \in B_{\PP}(f_0,t_{*})-f_0}|n^{-1}\sum_{i=1}^{n}f(X_i)\xi_i| \leq K_{2}\E_{x,\xi}\sup_{B_{\PP}(f_0,t_{*})-f_0} |n^{-1}\sum_{i=1}^{n}f(X_i)\xi_i|  + \frac{10t_{*}}{\sqrt{n}} + \frac{C\log(n)}{n},
\end{aligned}
\end{equation}
with probability of at least $0.99$ over  both $X_1,\ldots,X_n$ and $\vec\xi$.

As above, for $n$ large enough,  we can find an event $\cE_8 \subseteq \cE_1$ (where $\cE_1$ is defined in Eq.~\eqref{Eq:E1})  of probability at least  $0.98$ (over $X_1,\ldots,X_n$), such that for any $\Xn \in \cE_8 $, there exists an event $\cA_4(\Xn)$ of probability at least $0.98$ (over $\vec\xi$) such that \eqref{Eq:simEq2} holds. 
Then, similarly to the case of $\cE_2$, we will employ Lipschitz concentration for the middle term  in \eqref{Eq:simEq2}, for each $\Xn \in \cE_8$. To estimate the Lipschitz constant, recall that under  $\cE_1$ (more precisely, under the event of Lemma~\ref{Lem:Biso}), we also have that $$\|f-f_0\|_{n}^2  \leq s_1^2 \GaussAvgsF +10\GaussAvgsF \leq 11\GaussAvgsF$$ for all $f \in B_{\PP}(f_0,t_{*})$, under the choice $t_*$ in \eqref{Eq:tstar}. Then, using the fact that $\cA_4(\Xn)$ holds with probability of at least $0.98$, and Lemma \ref{Lem:Lipschtiz} with Lipschitz constant $\sup_{f \in B_{\PP}(f_0,t_{*})}\|f-f_0\|_{\PP_n} \leq \sqrt{11\GaussAvgsF}$, imply that for each $\Xn \in \cE_8$, the middle term in \eqref{Eq:simEq2} is within an additive factor of  $C_1\sqrt{\GaussAvgsF}n^{-1/2}$ from its expectation over $\vec{\xi}$.  Namely, we have for all $\Xn \in \cE_8$:
\begin{align*}
\E_{\xi}\sup_{f \in \cF-f_0}|n^{-1}\sum_{i=1}^{n}f(X_i)\xi_i| 
    \leq K_{2}\E_{x,\xi}\sup_{f \in \cF-f_0} |n^{-1}\sum_{i=1}^{n}f(X_i)\xi_i|  + \frac{C\sqrt{\GaussAvgsF}}{\sqrt{n}}. 
\end{align*}
 where we used the fact that $t_{*} \leq s_1\sqrt{\GaussAvgsF}$. The claim for $\cE_3$ follows by similar considerations that we used earlier. 
\end{proof}

\begin{proof}[Proof of Corollary \ref{Cor:Donsker}]
For any $\PP$-Donsker class we have with probability at least $0.9$ \citep[Chap. 5]{van2000empirical}
$$ \EmpGaussAvgs(\cF) \sim \GaussAvgsF \sim n^{-1/2}.$$
Then, by Corollary \ref{cor:fixed_design}, we have that
\[  
   \E\int (\hat{f}_n - f_0)^2d\PP_n \gtrsim n^{-1} .
\]
In order to prove the second part of the bound, we apply Theorem \ref{Thm:URandomDesign}, 
\[
    \E\int (\hat{f}_n - f_0)^2d\PP \gtrsim \max\{n^{-1},t_{n,\PP}(f_0,\cF)^2\}.
\]
The corollary will follow if we show that for any $f_0 \in \cF$, we have that $t_{n,\PP} \gtrsim 1$. To see this, we use \citep[Thm 5.11]{van2000empirical} that shows that for all $ t \geq 0$, we have
\[
    \GaussAvgs(B_{\PP}(f_0,t)) \lesssim n^{-1/2}\int_{0}^{t}u^{-\alpha/2}du \lesssim t^{\frac{-\alpha+2}{2}}n^{-1/2}.
\]
Since $\alpha \in (0,2)$, the right hand side is decreasing in $t$, therefore we know that if
\[
    \GaussAvgs(B_{\PP}(f_0,t_{*})) \gtrsim \GaussAvgsF \gtrsim n^{-1/2}
\]
then we have $t_{*} \gtrsim 1$. Hence, $t_{n,\PP}(f_0,\cF) \gtrsim 1$, and the claim follows.
\end{proof}
\begin{proof}[Proof of Corollary \ref{Cor:NonDonsker}]
For any non $\PP$-Donsker class we have with probability of at least $0.9$ \citep[Chap. 5]{van2000empirical}
$$  n^{-\frac{2}{2+\alpha}} \lesssim \EmpGaussAvgs(\cF) \sim \GaussAvgsF \lesssim n^{-\frac{1}{\alpha}}.$$
Then, by Corollary \ref{cor:fixed_design}, we have that
\[  
   \E\int (\hat{f}_n - f_0)^2d\PP_n \gtrsim n^{-\frac{4}{2+\alpha}},
\]
and the claim follows.
\end{proof}

\subsection{An example to  the tightness of Theorem \ref{Thm:URandomDesign} (a sketch) }\label{Sec:SimpleExample}
Let $\PP$ be the uniform density of $[0,1]$, and denote by $I(x_i,l_i)$ to be an interval with center $x_i$ and length $l_i$. For each $m \geq 0$ we define \begin{align*}
    \cF_{m}&:= \big\{ m^{-1/6}\sum_{i=1}^{m}\eps_i 1_{I(x_i,m^{-5/4})}: \forall x_1,\ldots,x_m \text{ s.t. } 1 \leq  j \ne k \leq m \ \ I(x_k,m^{-5/4}) \cap I(x_j,m^{-5/4}) = \emptyset, \\&  \quad \forall (\eps_1,\ldots,\eps_m) \in \{-1,1\}^m\big\}.\end{align*}
Now, we define $\cF:= \mathrm{conv}\{0,\{\cF_{m}\}_{m=1}^{\infty}\}$. Clearly, this family is uniformly bounded by one. Also, we assume that $f_0 = 0$.

Using a classical fact, we have that  with probability of at least $1-n^{2}$,  $$\max_{1\leq i \ne j \leq n}|X_j-X_i| \geq c\cdot(n\log n)^{-1},$$ and denote this event by $A$.
Clearly, for each $ \Xn \in A$, we can find a function $f_{\xi} \in \cF_{n}$ (that depends on $\Xn$ as well) such that
\begin{equation}\label{Eq:bb1}
    \inner{f_{\xi},\vec{\xi}}_n = n^{-1/6}\cdot n^{-1}\sum_{i=1}^{n}|\xi_i|.
\end{equation}
Also, note that under the event $A$,  $\hat{f}_n \notin \{\cF_{m}\}_{m=n+1}^{\infty}$. Therefore, one can easily show that 
\[
    \EmpGaussAvgs(B_{n}(f_0,n^{-1/6})) \sim n^{-1/6}.
\]
Now, denote by $C(n):= C_1(n\log(n))^{4/5} $ for $C_1$ that is large enough.  Note that any
$ \cF_{m} $ such that $ C(n) \leq m \leq n-1$, we can only place  $m$ intervals with length of at most $(c/2)\cdot (n\log(n))^{-1}$. Therefore, under the event $A$, each of these intervals has at most  one point. Hence, we have that
\[
    \max_{f_m \in \cF_{m}}\inner{f_m,\vec{\xi}}_n =  m^{-1/6}n^{-1}\max_{S \in \binom{n}{m},|S|=m} \sum_{i\in S}|\xi_i|.
\]
Now, for any fixed $ C(n)\leq m \leq n-1$, one can easily show by standard concentration inequalities that
\begin{equation}\label{Eq:bb2}
    \E\max_{f_m \in \cF_{m}}\inner{f_{m},\vec{\xi}}_n  \lesssim m^{-1/6}(m/n) + m^{-1/6}(m/n)\sqrt{\log(n/m)} \lesssim m^{-1/6}(m/n)\sqrt{\log(n/m)}.
\end{equation}
In the remaining case of $m \leq C(n)$, using some standard arguments, it can be shown that with probability of at least $1-n^2$ (over $X_1,\ldots,X_n$) the following holds: 
\begin{equation}\label{Eq:bb3}
\begin{aligned}
   \E_{\xi}\sup_{f_m \in \cF_m}\inner{f_m,\vec{\xi}}_n \sim \E_{x,\xi}\sup_{f_m \in \cF_m}\inner{f_m,\vec{\xi}}_n
    & \ll n^{-1/6}. 
\end{aligned}
\end{equation}
By using Eqs. \eqref{Eq:bb1},\eqref{Eq:bb2},\eqref{Eq:bb3}, one can show that with high probability  $\hat{f}_n \in B_n(f_0,Cn^{-1/6}),$ for some $C \geq 0$,  and also 
$$
    \GaussAvgsF \sim n^{-1/6}.
$$
Therefore, one can conclude that 
\[
     \E\int (\hat{f}_n - f_0)^2d\PP_n \sim n^{-\frac{1}{3}} \sim \GaussAvgsF^2,
\]
and 
\[
     \E\int (\hat{f}_n - f_0)^2d\PP \sim n^{-(\frac{1}{4}+\frac{1}{3})} \ll \GaussAvgsF^2 \sim  n^{-\frac{1}{3}}.
\]
Finally, it is easy to see that $t_{n,\PP}(f_0,\cF) \gtrsim n^{-(\frac{1}{8}+\frac{1}{6})}$, and therefore, by using the last equation 
\[
    \E\int (\hat{f}_n - f_0)^2d\PP \sim t_{n,\PP}(f_0,\cF)^2,
\]
and the claim follows.

\end{document}